\documentclass[11pt]{amsart}

\usepackage{palatino}
\usepackage{amsmath,amssymb,amsthm,amscd, amsxtra, amsfonts,graphicx,fancybox}
\usepackage[all]{xy}
\usepackage{amssymb}
\usepackage{latexsym}
\usepackage{amscd}
\usepackage{color}

\newtheorem{theorem}{Theorem}[section]
\newtheorem{corollary}{Corollary}[section]
\newtheorem{lemma}{Lemma}[section]

\newtheorem{remark}{Remark}[section]

\newtheorem{proposition}{Proposition}[section]

\def \a{\alpha }

\def \l{\lambda }

\begin{document}

\newcommand{\wta}{{\rm {wt} }  a }
\newcommand{\R}{\frak R}
\newcommand{\cV}{\mathcal V}
\newcommand{\cA}{\mathcal A}
\newcommand{\cL}{\mathcal L}
\newcommand{\J}{\mathcal H}
\newcommand{\G}{\mathcal G}
\newcommand{\wtb}{{\rm {wt} }  b }
\newcommand{\bea}{\begin{eqnarray}}
\newcommand{\eea}{\end{eqnarray}}
\newcommand{\be}{\begin {equation}}
\newcommand{\ee}{\end{equation}}
\newcommand{\g}{\frak g}
\newcommand{\hg}{\hat {\frak g} }
\newcommand{\hn}{\hat {\frak n} }
\newcommand{\h}{\frak h}
\newcommand{\U}{\mathcal U}
\newcommand{\hh}{\hat {\frak h} }
\newcommand{\n}{\frak n}
\newcommand{\Z}{\Bbb Z}
\newcommand{\N}{{\Bbb Z} _{> 0} }
\newcommand{\Zp} {\Z _ {\ge 0} }
\newcommand{\C}{\Bbb C}
\newcommand{\Q}{\Bbb Q}
\newcommand{\vak}{\bf 1}
\newcommand{\la}{\langle}
\newcommand{\ra}{\rangle}
\newcommand{\NS}{\bf{ns} }

\newcommand{\hf}{\mbox{$\tfrac{1}{2}$}}
\newcommand{\thf}{\mbox{$\tfrac{3}{2}$}}

\newcommand{\W}{\mathcal{W}}
\newcommand{\non}{\nonumber}
\def \l {\lambda}
\baselineskip=14pt
\newenvironment{demo}[1]%
{\vskip-\lastskip\medskip
  \noindent
  {\em #1.}\enspace
  }%
{\qed\par\medskip
  }

\def \l {\lambda}
\def \a {\alpha}

\keywords{vertex superalgebras, affine Lie algebras, Clifford
algebra, Weyl algebra, lattice vertex algebras,  critical level}
\title[]{
 A classification of irreducible Wakimoto modules for the affine Lie algebra $A_1 ^{(1)}$  }

  \subjclass[2000]{
Primary 17B69, Secondary 17B67, 17B68, 81R10}
\author{ Dra\v zen Adamovi\' c }

\date{}
\curraddr{Department of Mathematics, University of Zagreb,
Bijeni\v cka 30, 10 000 Zagreb, Croatia} \email {adamovic@math.hr}
\markboth{Dra\v zen Adamovi\' c} { }
\bibliographystyle{amsalpha}
\pagestyle{myheadings} \maketitle

\def \l {\lambda}
\def \a {\alpha}

\begin{abstract}
By using methods developed in Adamovi\' c (Comm. Math. Phys. 270 (2007) 141-161) we study  the
irreducibility of certain  Wakimoto modules for $\widehat{sl_2}$
at the critical level. We classify all $\chi \in {\C}((z))$ such
that the corresponding Wakimoto module $W_{\chi}$ is irreducible.
It turns out that zeros of  Schur polynomials play important rule
in the classification result.
\end{abstract}

\maketitle

\section{Introduction}
 In the representation theory of affine Kac-Moody Lie algebras,
represenations at the critical level belong to one of    the most
important cases. The Kac-Kazhdan conjecture for characters
motivates explicit realizations of  irreducible  highest weight
modules at the critical level. These representations can be
realized by using Wakimoto modules (cf. \cite{efren}, \cite{FF0},
\cite{FF1}, \cite{FP}, \cite{Scz}, \cite{W-mod}).  In \cite{A-2007} we
introduced an infinite-dimensional Lie superalgebra ${\cA}$
which is a certain limit of N=2 superconformal algebras obtained
by using Kazama-Suzuki mappings (cf. \cite{A3}, \cite{FST},
\cite{KS}). We also constructed a family of functors which send
irreducible ${\cA}$--modules to irreducible modules for the affine
Lie algebra $A_1 ^{(1)}$ at the critical level. By using this
construction we proved irreducibility of a large family of
Wakimoto modules $W_{\chi}$ parameterized by $\chi \in {\C}((z))$.
In this paper we shall completely solve the irreducibility problem
for Wakimoto modules $W_{\chi}$. We shall describe all $\chi \in
{\C}((z))$ such that $W_{\chi}$ is irreducible.

We first consider ${\cA}$--modules $\widetilde{F}_{\chi}$
constructed by using representations of  the infinite-dimensional
Clifford algebra and also parameterized by $\chi \in {\C}((z))$.
The functor $\mathcal{L}_0$ sends $\widetilde{F}_{\chi}$ to the
Wakimoto module $W_{-\chi}$ (cf. \cite{A-2007}). Then $W_{-\chi}$
is  irreducible $A_1 ^{(1)}$--module if and only if
$\widetilde{F}_{\chi}$ is irreducible ${\cA}$--module (cf.
Theorems \ref{ired-w-1} and \ref{red-w-1}). So we only need to
classify $\chi \in {\C}((z))$ such that $\widetilde{F}_{\chi}$ is
irreducible. By combining results from \cite{A-2007} and results
from the present paper, we obtain the following classification
result.

 \begin{theorem}
Assume that $\chi \in {\C}((z))$. Then the Wakimoto module
$W_{-\chi}$ is an irreducible $A_1 ^{(1)}$--module (resp.
$\widetilde{F}_{\chi}$ is irreducible ${\cA}$--module) if and only
if $\chi$ satisfies one of the following conditions:
\item[(i)] There is $p \in {\Zp}$, $p \ge 1 $ such that
$$ {\chi}(z) = \sum_{n=-p} ^{\infty} {\chi}_{-n} z ^{n-1} \in
{\C}((z)) \quad \mbox{and} \quad
\chi_p \ne 0. $$

\item[(ii)] $$ {\chi}(z) = \sum_{n=0} ^{\infty} {\chi}_{-n} z
^{n-1} \in {\C}((z)) \quad \mbox{and} \quad
\chi_0 \in \{1 \} \cup ({\C} \setminus {\Z}). $$
\item[(iii)] There is $\ell \in {\N}$ such that
$$ {\chi}(z) =  \frac{\ell +1}{z} + \sum_{n=1} ^{\infty} {\chi}_{-n} z ^{n-1} \in
{\C}((z))$$
and $S_{\ell}(-\chi) \ne 0$,
where $S_{\ell}(-\chi) = S_{\ell}(-\chi_{-1}, -\chi_{-2}, \dots)$
is a Schur polynomial.
\end{theorem}

We also  prove that when the Wakimoto module $W_{-\chi}$ is
reducible, then it contains an irreducible submodule.

Although the methods used in this paper can be mainly applied for
the affine Lie algebra $A_1 ^{(1)}$, we believe that the main
classification result can be extended for higher rank case. We
hope to study this problem in our future publications.

We would like to thank the referee for his valuable comments.

\section{ Clifford vertex superalgebras}

   The Clifford
algebra $CL$ is a complex associative algebra generated by
$$ \Psi^{\pm}(r) , \  r  \in \hf + {\Z},$$ and relations
\bea
&& \{\Psi^{\pm}(r) , \Psi^{\mp}(s) \} = \delta_{r+s,0}; \quad
 \{\Psi^{\pm}(r) , \Psi^{\pm}(s)\}=0
\nonumber
\eea
where $r, s \in {\hf}+ {\Z}$.

 Let $F$ be the irreducible $CL$--module generated by
 the
cyclic vector ${\vak}$ such that
$$ \Psi^{\pm} (r) { \vak} = 0 \quad
\mbox{for} \ \ r > 0 .$$
%
%A basis of $F$ is given by
%
%$$ \Psi^{- }({-n_1-{\hf}})  \cdots \Psi^{-}({-n_r-{\hf}})  \Psi^{+}({-k_1-{\hf}})  \cdots \Psi^{+}({-k_s-{\hf}})
% {\vak} $$
%
%where $n_i, k_i \in {\Zp}$,  $n_1 >n_2 >\cdots >n_r  $, $k_1 >k_2
%>\cdots
%>k_s $.

As a vector space,
$$F = \bigwedge ( \Psi^{-} (-n - \tfrac{1}{2} ), n \in {\Zp} ) \otimes  \bigwedge ( \Psi^{+} (-n - \tfrac{1}{2} ), n \in {\Zp} )$$
where $ \bigwedge (x_i, i \in I) $ denotes the exterior algebra with generators $x_i, i \in I$.

Define the following   fields on $F$
$$   \Psi^{+}(z) = \sum_{ n \in   {\Z}
 } \Psi^{+}(n+{\hf} )  z ^{-n- 1}, \quad  \Psi^{-} (z) = \sum_{ n \in {\Z}
 } \Psi^{-} (n+{\hf} )  z ^{-n-1}.$$

 The fields $\Psi^{+}(z)$ and $\Psi^{-}(z)$ generate on $F$  the
unique structure of a simple vertex superalgebra (cf.  \cite{K},
\cite{FB}).

Define the following Virasoro vector in $F$ :

$$ \omega ^{(f)} = \frac{1}{2} ( \Psi ^{+} (-\tfrac{3}{2} ) \Psi
^{-}(-\tfrac{1}{2}) +  \Psi ^{-} (-\tfrac{3}{2} ) \Psi
^{+}(-\tfrac{1}{2})) {\vak}.$$

Then the components of the field  $L^{(f)}(z) = Y(\omega^{(f)},z)
= \sum_{n \in {\Z} } L^{(f)}(n) z^{-n-2}$ defines on $F$ a
representation of the Virasoro algebra with central charge $1$.

Set $$J^{(f)}(z)= Y(\Psi^{+}(-{\hf})\Psi^{-}(-{\hf}){\vak},z)=
\sum_{n \in {\Z} } J^{(f)}(n) z^{-n-1}.$$

Then we have
$$ [J^{(f)}(n), \Psi^{\pm} (m+ {\hf})] = \pm \Psi^{\pm} (m+n +
{\hf}).$$
 Let $\widetilde{F} = \mbox{Ker}_F \Psi ^{-}(\hf)$
be the subalgebra of the vertex superalgebra $F$ generated by the
fields
$$ \partial \Psi^{+}(z)=\sum_{n \in {\Z} } -n \Psi^{+}(n - {\hf}) z^{-n-1} \ \ \mbox{and} \ \
 \Psi^{-}(z) = \sum_{ n \in {\Z} }   \Psi^{-}(n+ {\hf})
z^{-n-1}. $$

%
% ......... nova verzija..............
%

Then $\widetilde{F}$ is a simple vertex superalgebra and it is    $\hf {\Zp}$--graded with respect to the
operator $L^{f}(0)$. Let us describe the basis of $\widetilde{F}$.  A superpartition is a sequence $\lambda = (\lambda_n)_{n \in \N}$  in $ S \cup \{0\} $, $S \subset \Q_+$,  such that
$$ \lambda_1 > \lambda_2 > \cdots   \qquad \mbox{and} \quad \lambda_n = 0 \quad \mbox{for} \ n \ \mbox{sufficiently large}. $$
Define the length of partition by  $\ell (\lambda) = \max\{ n \ \vert \ \lambda_n \ne 0 \}$. If $\ell(\lambda) = \ell$ we write $\lambda= (\lambda_1, \dots, \lambda_\ell)$. Let $\phi$ denotes the
superpartition with all the entries being zero. Then we define $\ell (\phi) = 0$.

Let $\mathcal{P}$ be   the set of all superpartitions in $(\tfrac{1}{2} +  {\Zp})\cup \{ 0\} $ and $\overline{\mathcal{P}}$ be   the set of all superpartitions in $(\tfrac{3}{2} +  {\Zp} ) \cup \{0\} $.
Then we have $$\mathcal{P}= \cup_{r=0} ^ {\infty}  \mathcal{P}_r, \quad \overline{\mathcal{P} } = \cup_{r =0} ^{\infty}    \overline{\mathcal{P} }_r $$
where $\mathcal{P}_0   =   \overline{\mathcal{P} }_0 = \{ \phi \}$, and
\bea
&& \mathcal{P}_r   = \{ \lambda = (\lambda_1, \dots, \lambda_r) \in (\tfrac{1}{2} + \Z) ^r  \vert   \ \lambda_1 > \lambda_2 > \cdots > \lambda_r \ge 1/2 \}   \nonumber \\
 && \overline{\mathcal{P} }_r   = \{ \lambda = (\lambda_1, \dots, \lambda_r) \in (\tfrac{1}{2} + \Z) ^r  \vert \   \ \lambda_1 > \lambda_2 > \cdots > \lambda_r \ge 3/2 \} .\nonumber
  \eea

For  $\lambda = (\lambda_1, \dots, \lambda_r) \in \mathcal{P}_r  $, $\mu = (\mu_1, \dots, \mu_s) \in  \overline{\mathcal{P} }_s$
  we set
\bea && v_{\lambda, \mu} := \Psi^{- }({-\lambda_1})  \cdots
\Psi^{-}({-\lambda_r}) \Psi^{+}({-\mu_1})  \cdots
\Psi^{+}({-\mu_s})
 {\vak}  \nonumber \\
 && v_{\lambda, \phi} := \Psi^{- }({-\lambda_1})  \cdots
\Psi^{-}({-\lambda_r})
 { \vak}, \quad v_{\phi, \mu} :=   \Psi^{+}({-\mu_1})  \cdots
\Psi^{+}({-\mu_s})
 {\vak}, \nonumber \\
 &&  v_{\phi, \phi} = {\vak}. \nonumber
 \eea
 Then the set
\bea \label{baza-tilde} \{ v_{\lambda, \mu} \ \vert (\lambda, \mu) \in \mathcal{P} \times \overline{\mathcal{P} } \} \eea
 is a basis of $\widetilde{F}$.

 %Introduce the total order on the set  $\mathcal{P}$ (resp.   $\overline{\mathcal{P} }$ ) by
% $$ \lambda = (\lambda_1, \dots, \lambda_r) > \nu = (\nu_1, \dots, \nu_s)   $$
% if for some $i_0 \le \min \{r,s\}$
 %$$ \lambda_i = \nu_i, \ i=1, \dots, i_0 -1, \quad \lambda_{i_0} > \nu_{i_0}. $$
% or
%$$ \lambda_i = \nu_i, \ i=1, \dots, s \quad \mbox{and} \ r> s.$$

\section{ The vertex superalgebra ${\cV}$ and its modules}
\label{ver-def}
In this section we shall recall definition of the   vertex
superalgebra ${\cV}$ and certain results from \cite{A-2007}.
Let $M(0) = {\C}[{\gamma}^{+}(n), {\gamma}^{-}(n) \  \vert \ n <0
]$ be the commutative vertex algebra generated by the fields
$${\gamma} ^{\pm} (z) = \sum_{ n  < 0} {\gamma}^{\pm} (n) z ^{-n-1}. $$
(cf. \cite{efren}). Let $\chi^{\pm}(z) = \sum_{n \in {\Z} }
\chi^{\pm}_{n} z^{-n-1} \in {\C} ((z))$. Let $M(0, \chi^{+},
\chi^{-})$ denotes the $1$--dimensional irreducible $M(0)$--module
with the property that every element ${\gamma}^{\pm}(n)$ acts on
$M(0, \chi^{+}, \chi^{-})$ as multiplication by $\chi^{\pm}_n \in
{\C}$.

Let now ${\mathcal F}$ be the vertex superalgebra generated by the
fields $\Psi^{\pm} (z)$ and ${\gamma}^{\pm}(z)$. Therefore
${\mathcal F} = F \otimes M(0)$. As in \cite{A-2007}, denote by
${\cV}$ the vertex subalgebra of the  vertex superalgebra
${\mathcal F}$ generated by the following vectors
\bea
 \tau^{\pm} &=& (\Psi^{\pm} (-\tfrac{3}{2}) + {\gamma}^{\pm} (-1)
 \Psi^{\pm}(-\hf)) {\vak}, \label{def-tau} \\
 j &=& \frac{ {\gamma}^{+} (-1) - {\gamma}^{-}(-1)}{2} {\vak}, \label{def-j} \\
 \nu &=& \frac{ 2 {\gamma}^{+} (-1) {\gamma}^{-}(-1) + {\gamma}^{+}(-2) + {\gamma}^{-}(-2)}{4}
  {\vak} .
\label{def-nu}  \eea
The  vertex superalgebra structure on ${\cV}$     is generated by
the following fields \bea && G^{\pm} (z) = Y(\tau ^{\pm} ,z)
= \sum _{n \in {\Z} } G ^{\pm} (n+ {\hf} ) z ^{-n-2}, \label{polje-g} \\
 && S (z) = Y(\nu,z)
= \sum _{n \in {\Z} } S({n })  z ^{-n-2}, \label{polje-s} \\
&& T (z) = Y(j,z) = \sum _{n \in {\Z} } T(n )  z ^{-n-1}.
\label{polje-t} \eea

Denote by ${\cA}$ the  Lie superalgebra with  basis $ S(n), T(n),
{G} ^{\pm} (r), C$, $n\in {\Z}$, $ r\in {\hf} + {\Z}$ and
(anti)commutation relations  given by \bea && [S(m),S(n)] = [S(m),
T(n) ] = [S(m), G^{\pm} (r) ] =0,
\non \\
&& [T(m), T(n)]= [T(m), G^{\pm} (r)] =0, \nonumber \\ && [C,
S(m)]= [C, T(n)] = [C,
G^{\pm}(r)]=0, \nonumber \\
   && \{ G ^{+} (r),
G ^{-} (s) \} = 2 S({r+s}) +
(r-s) T({r+s}) + \tfrac{C}{3} ( r  ^ 2 - \tfrac{1}{4} ) \delta_{r+s,0}, \non \\
 && \{ G ^{+} (r), G ^{+} (s) \}= \{ G ^{-} (r), G ^{-} (s) \} = 0 \non
\eea for all $n \in {\Z}$, $r,s \in {\hf} + {\Z}$.

By using the commutator formulae for vertex superalgebras, we have
that the components of  fields (\ref{polje-g})-(\ref{polje-t})
satisfy the (anti)commutation relation for the Lie superalgebra
${\cA}$ such that the central element $C$ acts as multiplication
by $c=-3$.
%
%dodano
%
%
Let ${\cV} ^{com}$ be the commutative vertex subalgebra of ${\cV}$ generated by the fields $T(z)$ and $S(z)$. Clearly, ${\cV} ^{com} \cong M_T(0) \otimes M_S(0)$, where
$M_T(0)$ (resp. $M_S(0)$ ) is the  subalgebra of ${\cV} ^{com}$ generated by the field $T(z)$ (resp. $S(z)$).

Recall that ${\cV}$ admits the following $\Z$--graduation:
\bea
 {\cV}  & = & \bigoplus_{m \in {\Z} } {\cV} ^{m} \nonumber \\
{\cV} ^{m}  & = & \mbox{span}_{\C} \{ G ^{+} (-n_1 - \tfrac{3}{2}) \cdots G ^{+} (-n_r - \tfrac{3}{2}) G ^{-} (-k_1 - \tfrac{3}{2}) \cdots G ^{-} (-k_s - \tfrac{3}{2}) w \vert \nonumber \\
 & & w \in {\cV} ^{com}, n_i, k_j \in {\Zp}, r-s = m \}. \nonumber
\eea
 Now   we shall consider a family of
 irreducible ${\cV}$--modules.

For ${\chi}^{+}, {\chi}^{-} \in {\C}((z))$ we set
$F({\chi}^{+}, {\chi}^{-}) :=F \otimes M(0,\chi^{+},\chi^{-})$.

Then $F({\chi}^{+}, {\chi}^{-})$ is a module for the vertex
superalgebra ${\cV}$, and therefore for the Lie superalgebra
${\cA}$.

Since $ M(0,\chi^{+},\chi^{-})$ is one-dimensional, we have that
as  a vector space

\bea  \label{identification-prva} F({\chi}^{+}, {\chi}^{-}) \cong
F \cong { \bigwedge} (\Psi^{\pm} (-i-\hf) \ \vert \ i \ge 0 ).
\eea

Now let ${\chi} (z)  \in {\C}((z))$. Define:
$$ \widetilde{F}_{\chi} := \widetilde{F} \otimes M(0,0,{\chi}) \subset F ( 0,{\chi}).$$
The operator $J^{f}(0)$ acts semisimply on $\widetilde{F}_{\chi}$
and it defines the   following $\Z$--gradation
$$ \widetilde{F}_{\chi} = \bigoplus_{j \in \Z} \widetilde{F}_{\chi} ^j, \quad
\widetilde{F}_{\chi} ^j = \{ v \in \widetilde{F}_{\chi} \ \vert \
J^{f}(0) v = j v \}. $$

 The ${\cA}$--module
structure on $\widetilde{F}_{\chi}$ is uniquely determined by the
following action of the Lie superalgebra ${\cA}$ on
$\widetilde{F}$:
\bea
G^{+}(i-\hf) &=& -i \Psi^{+}(i-\hf) \label{++djelovanje1} , \\
G^{-}(i-\hf) &=&- i \Psi^{-}(i-\hf) + \sum_{k=-p} ^{\infty}
{\chi}_{-k} \Psi^{-} (k+i -{\hf}). \label{--djelovanje1}
 \eea

Now we shall first  recall the following  irreducibility result:

\begin{proposition} [\cite{A-2007}, Proposition 5.2] \label{ired1}
Assume that $p \in {\Zp}$ and that
$$ {\chi}(z) = \sum_{n=-p} ^{\infty} {\chi}_{-n} z ^{n-1} \in {\C}((z))$$
satisfies the following conditions
\bea
&& {\chi}_p \ne 0,   \label{uvjet-prvi} \\
&& {\chi}_0 \in  \{1\} \cup \left({\C} \setminus {\Z}\right)  \ \
\mbox{if} \ p=0 \label{uvjet-drugi}. \eea
Then
$\widetilde{F}_{\chi}$ is an irreducible ${\cV}$--module.
\end{proposition}

\section{Schur polynomials and irreducibility of
$\widetilde{F}_{\chi}$}

In this section  we shall extend the irreducibility result from
Proposition \ref{ired1}. We shall always assume that $ \chi(z)$
has the form
\bea &&{\chi}(z) =  \frac{\ell +1}{z} + \sum_{n=1} ^{\infty}
{\chi}_{-n} z ^{n-1} \in {\C}((z)), \label{oblik}\eea
where $\ell \in {\Z}$.
Then the ${\cA}$--module structure on $\widetilde{F}_{\chi}$ is
uniquely determined by the following action of the Lie
superalgebra ${\cA}$ on $\widetilde{F}$:
\bea
G^{+}(i-\hf) &=& -i \Psi^{+}(i-\hf) \label{++djelovanje} , \\
G^{-}(i-\hf) &=&({\ell}+1- i) \Psi^{-}(i-\hf) + \sum_{n=1}
^{\infty} {\chi}_{-n} \Psi^{-} (n+i -{\hf}). \label{--djelovanje}
 \eea

By Proposition \ref{ired1} we know that if $\ell$ is generic or
$\ell =0$, then $\widetilde{F}_{\chi}$ is an irreducible module.
We shall consider the case when $\ell \in \N$,  and find a
sufficient condition on $\chi(z)$ so that $\widetilde{F}_{\chi}$
is irreducible.

For every $ s \in {\N}$, we define $$\Omega_{s} =\Psi^{+} (-s
-\hf) \Psi ^{+} (-s + \hf) \cdots \Psi^{+} (-\tfrac{3}{2}) {\vak}
\in \widetilde{F}_{\chi}.
$$

We shall need the following lemma. The proof will use only the  action of the operators $G^{+}(i -\hf)$, $ i \in \Z$.

\begin{lemma} \label{omega}
Assume that $U \subset \widetilde{F}_{\chi}$ is any  submodule, $U
\ne \{0 \}$. Then there is $s \in {\N}$ such that
$$\Omega_s  \in U.$$
\end{lemma}
\begin{proof}
For $\lambda \in \mathcal{P}$ and $t \in \overline{\mathcal{P}}$, we set $$G^{+}_{\lambda} = \left\{ \begin{array}{cc}
                                         1&  \mbox{if} \  {\lambda} = \phi  \\
                                          G ^+ ( {\lambda}_1) \cdots G^{+} (   {\lambda}_r )  & \quad \mbox{if} \  {\lambda} = ( {\lambda}_1, \dots,  {\lambda}_r)
                                        \end{array}  \right.$$

             $$G^{+}_{-t} = \left\{ \begin{array}{cc}
                                         1&  \mbox{if} \  t = \phi  \\
                                          G ^+ ( -t_1) \cdots G^{+} ( -t_r )  & \quad \mbox{if} \  t = ( t_1, \dots,  t_r)
                                        \end{array}  \right.$$

Let $v \in U$, $v \ne 0$. Then $v$ has unique decomposition
$$ v = \sum_{ (\lambda, \mu) \in \mathcal{P} \times \overline{\mathcal{P}} } C_{\lambda, \mu} v_{\lambda,\mu} \quad (C_{\lambda,\mu} \in {\C} ) $$
in the basis (\ref{baza-tilde}). Let $\ell = \max \{ \ell (\lambda) \ \vert  \ C_{\lambda, \mu} \ne 0 \}$. We can choose   $(\bar {\lambda}, \bar{\mu}) \in \mathcal{P} \times \overline{\mathcal{P}} $ such that
\begin{itemize}
 \item[(1)] $C_{\bar {\lambda}, \bar{\mu}} \ne 0 $, $  \ell (\bar{\lambda}) = \ell $,
\item[(2)]$\ell (\overline{\mu})=\ell_1 = \min \{\ell(\mu)  \vert \mu \in T_1 \}$ where  $ T_1= \{ \mu \in \overline{\mathcal{P} }  \ \vert  \ C_{\overline{\lambda}, \mu} \ne 0 \} $.
\end{itemize}

If $T_1 = \{ \phi \}$, we set $f = G^{+}_ {\overline{\lambda} }$. Otherwise, let
$s \in \N$ be such that
$$ \max \{ \mu_1 \ \vert \ \mu = (\mu_1, \dots, \mu_l) \in T_1 \} = s+ \tfrac{1}{2}. $$

If  $\overline{\mu} = (\overline{\mu}_1, \dots, \overline{\mu}_{\ell_1}) $ and $ 0 < \ell_1 = \ell(\overline{\mu}) < s$, there are unique  $ t _1 > \cdots > t_p $, $p = s- \ell_1$, such that
$$ \{ t_1, \dots, t_p\} =   \{ \tfrac{3}{2}, \dots, s+ \tfrac{1}{2}\} \setminus  \{ \overline{\mu}_1, \dots, \overline{\mu}_{\ell_1} \}.
 $$
Define now $t \in \overline{\mathcal{P}}$ in the following way:
 $$ t = \left\{ \begin{array}{ccc}
                                          \phi &  \mbox{if} \  & \ell_1 = s   \\
                                           (s+1/2, \dots, 3/2)    & \quad \mbox{if} \  &  \ell_1= 0  \\
                                          (t_1, \dots, t_p)  & \quad \mbox{if }  &0 < \ell_1 < s
                                        \end{array}  \right.   .$$

 Then we set
 $$ f = G^{+}_{-t} G^{+}_{ \overline{\lambda} }. $$

By construction we have that $G^{+}_{ \overline{\lambda} }$ annihilates basis vectors $v_{\lambda, \mu}$ such that $\ell(\lambda) \le \ell$, $\lambda \ne \overline{\lambda}$, and $G^{+}_{-t}$ annihilates all $v_{\overline{\lambda}, \mu}$\ , where $\mu \in T_1 \setminus \{ \overline{\mu} \}$. Therefore,
\bea  && f v_{\lambda, \mu}  = 0 \quad \mbox{if} \quad C_{\lambda, \mu} \ne 0 \ \mbox{and} \ (\lambda, \mu ) \ne (\overline{\lambda}, \overline{\mu}), \nonumber \\
&&  f v_{\overline{\lambda}, \overline{\mu} } =\nu \Omega_{s} \quad (\nu\ne 0) \quad \mbox{if} \ T_1 \ne \{\phi\}, \nonumber \\
&&   f v_{\overline{\lambda}, \overline{\mu}  } = \nu_1 {\vak} \quad (\nu_1 \ne 0) \quad \mbox{if} \ T_1 = \{\phi\}. \nonumber \eea
The proof follows.
\end{proof}

In order to present  new irreducibility criterion, we shall first
recall the definition of Schur polynomials.

 Define the Schur polynomials $S_{r}(x_{1},x_{2},\cdots)$
 in variables $x_{1},x_{2},\cdots$ by the following equation:
\begin{eqnarray}\label{eschurd}
\exp \left(\sum_{n= 1}^{\infty}\frac{x_{n}}{n}y^{n}\right)
=\sum_{r=0}^{\infty}S_{r}(x_1,x_2,\cdots)y^{r}.
\end{eqnarray}

We shall also use the following formula for Schur polynomials:
\bea  \label{det-schur} S_{r}(x_{1},x_{2},\cdots) &=\frac{1}{r !}&
\det \left(\begin{array}{ccccc}
  x_{1} & x_{2} & \cdots  & \  & x_{r} \\
  -r +1 &x_{1}  & x_{2} & \cdots &  x_{r- 1} \\
  0 & -r+2 & x_{1} & \cdots & x_{r- 2} \\
  0 & \ddots & \ddots & \ddots &  \\
  0 & \cdots & 0 & -1 & x_{1} \\
\end{array}
\right) \eea

\begin{lemma} \label{schur-1}We have
$$G^- (\hf) \cdots G^- (\ell -\hf) \Omega_{\ell} =(-1) ^{\ell} {\ell} ! \ S_{\ell} (-\chi)
{\vak}
$$
where  $S_{\ell} (-\chi) = S_{\ell}(-\chi_{-1}, \dots,
-\chi_{-\ell}, \dots )$.
\end{lemma}
\begin{proof} By using action (\ref{--djelovanje}) we get:
\bea
& & G^- (\hf) \cdots G^- (\ell -\hf) \Omega_{\ell} \nonumber \\
= && \left( \ell \Psi ^{-} (  \tfrac{1}{2})+\chi_{-1} \Psi^{-}
(\tfrac{3}{2}) + \cdots +\chi_{-\ell} \Psi^{-} (\ell + \hf)
\right)  \cdots \nonumber \\
&&  \left( 2 \Psi ^{-} ( \ell - \tfrac{3}{2})+ \chi_{-1} \Psi^{-}
(\ell - \hf) + \chi_{-2} \Psi^{-} (\ell + \hf) \right)  \left(
\Psi^{-} (\ell - \hf) + \chi_{-1} \Psi^{-} (\ell + \hf) \right)
\Omega_{\ell}
\nonumber \\
=&&\det  \left(
\begin{array}{ccccc}
  \chi_{-1} & \chi_{-2} & \cdots  & \  & \chi_{-\ell} \\
  \ell -1 &\chi_{-1}  & \chi_{-2} & \cdots &  \chi_{-\ell +1} \\
  0 & \ell-2 & \chi_{-1} & \cdots & \chi_{-\ell +2} \\
  0 & \ddots & \ddots & \ddots &  \\
  0 & \cdots & 0 & 1 & \chi_{-1} \\
\end{array}%
\right) \nonumber \\
= && (-1) ^{\ell} {\ell} ! \ S_{\ell}(-\chi_{-1}, \dots,
-\chi_{-\ell}, \dots ) {\vak}. \nonumber
\eea (Here we use elementary properties of determinants and
formula (\ref{det-schur}) for  Schur polynomials). \end{proof}

\begin{proposition} \label{ired-a-1}
Assume that $\ell \in \N$,
$$ {\chi}(z) =  \frac{\ell + 1}{z} + \sum_{n=1} ^{\infty} {\chi}_{-n} z ^{n-1} \in
{\C}((z))$$
such that
\bea
&& S_{\ell} (-\chi)  \ne 0. \nonumber
\eea Then
$\widetilde{F}_{\chi}$ is an irreducible ${\cV}$--module.
\end{proposition}
\begin{proof}
First we shall prove that the vacuum vector is a cyclic vector for
 the $U(\cA)$--action, i.e.,
 \bea
\label{ciklic} U({\cA}). {\vak} = \widetilde{F}.
 \eea

Take an arbitrary  basis element
\bea &&v= \Psi^{+ }({-n_1-{\hf}})  \cdots \Psi^{+}({-n_r-{\hf}})
\Psi^{-}({-k_1-{\hf}})  \cdots \Psi^{-}({-k_s-{\hf}})
 {\vak}  \in \widetilde{F} ,  \label{baza}\eea
 where $n_i, k_i \in {\Zp}$,  $n_1 >n_2 >\cdots
>n_r > 0 $, $k_1
>k_2
>\cdots
>k_s \ge 0 $.

Let $N \in {\Zp}$ such that $N \ge k_1$. By using
(\ref{--djelovanje}) we get that
$$ G^{-}(- N - {\hf}) \cdots G^{-}(-\thf) G^{-}(-\hf) {\vak} = C \Psi^{-}(-N -{\hf}) \cdots
\Psi ^{-}(-\thf)\Psi^{-} (-\hf) {\vak},
$$
where
$$C=
  (\ell +1) (\ell +2) \cdots (\ell  + N+1)  $$
So $C\ne 0$, and we have that
$$
\Psi^{-}(-N -{\hf}) \cdots \Psi ^{-}(-\thf)\Psi^{-} (-\hf) {\vak}
\in U({\cA}) .{\vak} .$$
By using this fact and the action of elements $G^{+}(i-{\hf})$, $i
\in {\Z}$, we obtain that $v \in U(\cA).{\vak}$. In this way  we
proved (\ref{ciklic}).

It is enough to prove that every vector $u \in
\widetilde{F}_{\chi}$ is cyclic. So let $U = U(\cA). u$. By using
Lemma \ref{omega} we have that there is $s \in {\N}$ such that
$\Omega_s \in U$. Assume that
 $ s > \ell$. Then clearly
\bea
 && G ^{-}( \ell + \tfrac{3}{2}) \cdots G ^{-} ( s + \hf)
\Omega_s = C_1 \Omega_{\ell} \label{rel-djel-1}
\eea
for certain non-zero constant $C_1$. Similarly, if $s < \ell$ we
see that
\bea
 &&  G^{+} (\ell + \hf) \cdots  G^{+} (s + \tfrac{3}{2})
\Omega_s = C_2 \Omega_{\ell}, \quad (C_2 \ne 0).
\label{rel-djel-2}
\eea
Therefore we conclude that $\Omega_{\ell} \in U$.

Applying  Lemma \ref{schur-1} we get
$$G^- (\hf) \cdots G^- (\ell - \hf) \Omega_{\ell} = \nu {\vak}, \quad (\nu \ne 0).$$
Thus  ${\vak} \in U= U(\cA). u$. Now relation (\ref{ciklic}) gives
that $u$ is a cyclic vector in $\widetilde{F}_{\chi}$.
 The proof follows. \end{proof}

\begin{proposition} \label{red-1}
Assume that $\ell \in {\N}$ and  $S_{\ell}(-\chi) = 0$.

\item[(i)] Then $U_{\chi} = U({\cA}) . \Omega_{\ell}$ is a proper
submodule of $\widetilde{F}_{\chi}$. In particular,
$\widetilde{F}_{\chi}$ is reducible.

\item[(ii)] $U_{\chi}$ is an  irreducible $\cV$--module.
\end{proposition}
\begin{proof}
Assume that $S_{\ell} (-\chi)= 0$. Define
\bea w &=& G^{-}(\tfrac{3}{2}) \cdots G^{-} (\ell - \hf)
\Omega_{\ell} \nonumber \\
&=& ((-1) ^{\ell-1} ( \ell -1)! \Psi^{+} (-\ell - \hf ) + a_1 \Psi^{+} (-\ell
+ \hf )+ \cdots + a_{\ell -1} \Psi^{+} (-\tfrac{3}{2} ) ) {\vak}
\nonumber
\eea
where $a_1, \dots , a_{\ell-1}$ are certain complex numbers.

Therefore $w \ne 0$. By using Lemma \ref{schur-1}, the assumption
$S_{\ell} (-\chi)= 0$  and the definition of $w$ we get
$$ G^{\pm}(n-\hf) w = 0 \quad \mbox{for} \ \ n \in {\N}. $$
One can easily show that
$$ G^{+}( -\ell + \hf) \cdots   G^+(-\tfrac{3}{2}) w = C
\Omega_{\ell} \quad (C\ne 0), $$
which implies that $U_{\chi} = U(\cA). w$. Every element of
$U_{\chi}$ is a linear combination of vectors
\bea
\label{span-vect} && G^{-}(-n_1 - \hf) \cdots G^{-}(-n_r - \hf)
G^{+} (-m_1-\hf) \cdots G^{+}(-m_s -\hf) w,
\eea
for $n_i, m_i \in {\Zp}$, $ n_1 > n_2 > \cdots > n_r$, $m_1 > m_2
> \cdots > m_s$.
But a vector  (\ref{span-vect}) is either zero (if $G^{+}
(-m_1-\hf) \cdots G^{+}(-m_s -\hf) w = 0$) or has  the following
non-trivial summand of lowest degree in $\widetilde{F}$ (with
respect to $L^f(0)$)
$$ C  \Psi^{-}(-n_1 - \hf) \cdots \Psi^{-}(-n_r - \hf)
\Psi^{+} (-m_1-\hf) \cdots \Psi^{+}(-m_s -\hf)  w
$$
where $C \ne 0$.  From this one gets that ${\vak} \notin U_{\chi}$.
Therefore $\widetilde{F}_{\chi}$ is a reducible module with the
proper submodule $U_{\chi}$. This proves assertion (i).

Assume now that $U \subset U_{\chi}$ is a non-zero submodule. Then
Lemma \ref{omega} implies that  there is $s \in {\Zp}$ such that
$\Omega_s \in U$. By using relations (\ref{rel-djel-1}) and
(\ref{rel-djel-2}) from the proof of Proposition \ref{ired-a-1} we
see that $\Omega_{\ell} \in U$. Therefore  $U = U(\cA)
\Omega_{\ell} = U_{\chi}$ and $U_{\chi}$ is an  irreducible
${\cA}$--module. This proves assertion (ii). \end{proof}

 \begin{proposition} \label{red-2}
Assume that $ \ell \in {\Z}$, $\ell < 0$.
\item[(i)] $\widetilde{F}_{\chi}$ is reducible and $J_{\chi} =
U(\cA). 1$ is its  proper submodule.

\item[(ii)] $J_{\chi} $ is an irreducible $\cV$--module.
 \end{proposition}
\begin{proof} Let $q = - \ell  -1$. Then
$$ G^{}(n-\hf) =- (q +n) \Psi^{-}(n-\hf) +  \sum_{n=1}
^{\infty} {\chi}_{-n} \Psi^{-} (n+i -{\hf}).$$
By using similar arguments as in the proof of Proposition
\ref{red-1} one can see that $\Psi^{-}(-q -\hf) {\vak} \notin
J_{\chi}$ which gives reducibility of $\widetilde{F}_{\chi}$. The
proof  that submodule  $J_{\chi} $ is irreducible is completely
analogous to that of Proposition \ref{red-1} (ii). \end{proof}

Note that $U_{\chi}$ and $J_{\chi} $ are $\Z$--graded ${\cV}$--modules with respect to $J^{f}(0)$:
\bea
U_{\chi} & = & \bigoplus_{i \in {\Z} } U_{\chi} ^{i}, \quad  U_{\chi} ^{i} = \{ v \in U_{\chi} \vert \ J^{f}(0) v = i v \}, \label{grad-U}  \\
J_{\chi} & = & \bigoplus_{i \in {\Z} } J_{\chi} ^{i}, \quad  J_{\chi} ^{i} = \{ v \in J_{\chi} \vert \ J^{f}(0) v = i v \}. \label{grad-J}
\eea

\vskip 5mm

 Now we are able to classify $\chi \in {\C}((z))$ such that
 $\widetilde{F}_{\chi}$ is irreducible. We have proved the
 following classification result.

 \begin{theorem} \label{irreducibility-1}
Assume that $\chi \in {\C}((z))$. Then the ${\cV}$--module
$\widetilde{F}_{\chi}$ is irreducible if and only if $\chi$
satisfies one of the following conditions:
\item[(i)] There is $p \in {\Zp}$, $p \ge 1 $ such that
$$ {\chi}(z) = \sum_{n=-p} ^{\infty} {\chi}_{-n} z ^{n-1} \in
{\C}((z)) \quad \mbox{and} \quad
\chi_p \ne 0. $$

\item[(ii)] $$ {\chi}(z) = \sum_{n=0} ^{\infty} {\chi}_{-n} z
^{n-1} \in {\C}((z)) \quad \mbox{and} \quad
\chi_0 \in \{1 \} \cup ({\C} \setminus {\Z}). $$
\item[(iii)] There is $\ell \in {\Zp}$ such that
$$ {\chi}(z) =  \frac{\ell +1}{z} + \sum_{n=1} ^{\infty} {\chi}_{-n} z ^{n-1} \in
{\C}((z))$$
and $S_{\ell}(-\chi) \ne 0$.
\end{theorem}

\section{Wakimoto modules}

We shall first recall the definition of the Wakimoto modules at
the critical level (cf. \cite{efren}, \cite{W-mod}).

 The Weyl vertex
algebra $W$ is generated by the fields
$$ a(z)   = \sum_{n \in {\Z} } a(n) z^{-n-1}, \ \ a^{*}(z) =  \sum_{n \in {\Z} } a^{*}(n)
z^{-n}, $$
whose  components satisfy the commutation relations for the
infinite-dimensional Weyl algebra

 $$[a(n), a(m)] = [a^{*}(n), a^{*}(m)] = 0, \quad [a(n),
a^{*}(m)] = \delta_{n+m,0} .$$

  Assume that ${\chi}(z) \in {\C}((z))$.

  On the vertex algebra $W$ exists the structure of the $A_1
^{(1)}$--module at the critical level defined by

  \bea
e(z)&=& a(z), \nonumber \\
h(z) &=& - 2 :  a^{*}(z) a(z): - {\chi}(z)           \nonumber \\
f(z) & =&   - : a^{*} (z)  ^{2} a(z) : -2 \partial_{z} a^{*} (z) -
a^{*}(z) {\chi}(z) . \nonumber \eea

 This module is called the  Wakimoto module and it is denoted by $W_{-\chi}$.

 Let $F_{-1}$ be
the lattice vertex superalgebra $V_{L}$ associated to the lattice
$L={\Z}\beta$, where $\la \beta, \beta \ra = -1$ (cf.
\cite{A3},\cite{K}, \cite{LL}). Then $F_{-1}$ has the following
$\Z$--gradation (cf. \cite{A-2007}):
$$F_{-1} = \bigoplus_{j \in {\Z} } F_{-1} ^{j}, \quad F_{-1} ^{j}
= \{ v \in F_{-1} \ \vert \ \beta(0) v = - j v \}. $$

In \cite{A-2007}, we constructed   mappings ${\mathcal L}_s$, $s
\in {\Z}$, from the category of ${\cV}$--modules to the category
of $A_1 ^{(1)}$--modules at the critical level. Let $V_{-2}(sl_2)$ denotes the universal affine vertex algebra for $A_1 ^{(1)}$ at the critical level, and $M_T (0)$ be the commutative subalgebra of ${\cV}$ generated by the field $T(z)$.

\begin{theorem}[\cite{A-2007}, Theorem 6.2] \label{a-6.2}Assume that $U$ is a ${\cV}$--module which admits the following graduation:
$$ U = \bigoplus_{j \in {\Z} } U^{j}, \quad {\cV} ^{i} \cdot U^{j} \subset U^{i+j}. $$
Then $$ U \otimes F_{-1} = \bigoplus_{s \in {\Z} } \mathcal{L}_s (U)  \quad   \mathcal{L}_s (U)  = \bigoplus_{i \in {\Z}} U^{i } \otimes F_{-1} ^{-s + i}$$
and each $\mathcal{L}_s (U)$ is an $V_{-2}(sl_2) \otimes M_T(0)$--module. If $U$ is irreducible, then  $\mathcal{L}_s (U)$ is an irreducible $A_1 ^{(1)}$--module at the critical level.
\end{theorem}

 In particular, the
map ${\mathcal L}_0$ sends ${\cV}$--module $ \widetilde{F}_{\chi}$
to the Wakimoto module $W_{-\chi}$ and
$$ W_{-\chi} \cong \mathcal{L}_0 (\widetilde{F}_{\chi} ) = \bigoplus_{j \in {\Z} } \widetilde{F} _{\chi} ^
j \otimes F_{-1} ^{j}. $$

Recall first:

\begin{theorem} \label{ired-w-1} (\cite{A-2007}) Assume that  $
\widetilde{F}_{\chi}$ is an irreducible ${\cV}$--module. Then
$W_{-\chi}$ is irreducible $A_1 ^{(1)}$--module at the critical
level.
\end{theorem}

In the case of Wakimoto modules the converse is also true.
\begin{theorem} \label{red-w-1}
Assume that $ \widetilde{F}_{\chi}$ is reducible. Then the
Wakimoto module $W_{-\chi}$ is also reducible.
\end{theorem}
\begin{proof}Assume that $N \subset \widetilde{F}_{\chi}$ is any
proper submodule. Take  $s \in \N$ such that $\Omega_s \in N$ (cf.
Lemma \ref{omega}) and define $U = U(\cA). \Omega_s \subseteq N$.
Then $U$ admits the $\Z$--gradation
$$ U = \bigoplus_{j \in {\Z} } U ^{j} $$
where
$$ U ^{j} = \{ v \in U \ \vert \ J ^{f} (0) v = j v
\} \subset \widetilde{F}_{\chi} ^{j}.$$
Then by using Theorem \ref{a-6.2} we conclude that
$$\mathcal{L}_0 (U) = \bigoplus_{j \in {\Z} } U ^j
\otimes F_{-1} ^j $$
is an $A_{1} ^{(1)}$--module and it is a proper submodule of the
Wakimoto module $W_{-\chi}$. The proof follows. \end{proof}

\begin{corollary}
The Wakimoto module $W_{-\chi}$ is irreducible if and only if
$\chi(z) \in {\C}((z))$ satisfies one of the conditions (i)-(iii)
of Theorem \ref{irreducibility-1}.
\end{corollary}

In the case when the module $W_{-\chi}$ is reducible, it contains
irreducible submodule.
\begin{corollary} \label{reduc-struktura}
Let $\chi(z) = \frac{\ell+1}{z} + \sum_{n=1} ^{\infty}{\chi_{-n}}
z ^{n-1}$ and  $\ell \in {\Z}$.

\item[(i)]Assume that $ \ell \in {\N}$ and $S_{\ell} (-\chi) = 0$.
Then $$ \mathcal{L}_0(U_{\chi})= \bigoplus_{i \in {\Z}} U_{\chi} ^i \otimes F_{-1} ^i $$ is an irreducible submodule of
$W_{-\chi}$.

\item[(ii)] Assume that  $\ell < 0$. Then
$$\mathcal{L}_0(J_{\chi}) =  \bigoplus_{i \in {\Z}} J_{\chi} ^i \otimes F_{-1} ^i$$ is an irreducible submodule of $W_{-\chi}$.
\end{corollary}
\begin{proof}  Propositions \ref{red-1} and \ref{red-2} imply that
$U_{\chi}$ and $J_{\chi}$ are irreducible  ${\cV}$-modules  which
are $\Z$ graded with   graduations (\ref{grad-U}) and  (\ref{grad-J}). Then  Theorem \ref{a-6.2} implies that ${\mathcal L}_0(U_{\chi})$ and
$\mathcal{L}_0(J_{\chi})$ are irreducible $A_{1}^{(1)}$--modules.
The proof follows. \end{proof}

\begin{remark}
In the case of reducible Wakimoto modules from Corollary \ref{reduc-struktura} one can consider the action of $sl_2$ on $W_{-\chi}$ and the maximal $sl_2$--integrable submodule
$W_{-\chi} ^{int}$ of $W_{-\chi}$. It is clear that $W_{-\chi} ^{int}$ is a proper $A_{1} ^{(1)}$-- submodule of $W_{-\chi}$. By combining our results and  the results from \cite{FG} and \cite{ACM} one can easily show that
 $\mathcal{L}_0(U_{\chi}) = W_{-\chi} ^{int}  $ when $ \ell > 0$ (resp. $\mathcal{L}_0(J_{\chi}) = W_{-\chi} ^{int}$ when $ \ell < 0$). So our method shows that the maximal integrable submodule of the Wakimoto module $W_{-\chi}$ is irreducible $\widehat{sl_2}$--module at the critical level.
\end{remark}

\begin{remark}
It is interesting to look at the case $\ell =1$ and $ \chi(z) =
\frac{2}{z} +
 \sum_{n=1} ^{\infty} \chi_{-n} z ^{n-1}$. Then the Wakimoto
 module $W_{-\chi}$ is irreducible  if $\chi_{-1} \ne 0$ and
 reducible if $\chi_{-1} = 0$.
 The reducible Wakimoto modules
 $W_{-\chi}$ such that
  $$ \chi(z) = \frac{2}{z} +
 \sum_{n=2} ^{\infty} \chi_{-n} z ^{n-1}$$
  (i.e., $\chi_{-1} = 0$) were studied in \cite{FFR}.
\end{remark}


\begin{thebibliography}{Lia}



\bibitem [A1] {A3}  D. Adamovi\' c, Representations of the $N=2$
superconformal vertex algebra, Internat. Math. Res. Notices  {\bf
2} (1999) 61-79.







\bibitem [A2]{A-2007} D. Adamovi\' c, Lie superalgebras and
irreducibility of certain $A_1^{(1)}$--modules at the critical
level, Comm. Math. Phys. 270 (2007) 141-161


\bibitem[ACM]{ACM} T. Arakawa, D. Chebotarov, F. Malikov, Algebras of twisted chiral differential operators
and affine localization of $\g$-modules,  Selecta mathematica, new series, vol.17, no. 1, 1-46, 2011, arXiv:0810.4964

\bibitem[F]{efren} E. Frenkel, Lectures on Wakimoto modules,
opers and the center at the critical level, Adv. Math 195 (2005)
297-404.

\bibitem [FB]{FB} E. Frenkel and D. Ben-Zvi, Vertex algebras and
algebraic curves, Mathematical Surveys and Monographs; no. 88,
AMS, 2001.




\bibitem[FF1]{FF0} B. Feigin and E. Frenkel,  Representations of
affine Kac--Moody algebras and bosonization, in {\em Physics and
mathematics of strings}, pp. 271--316, World Scientific, 1990.



\bibitem[FF2]{FF1} B. Feigin and E. Frenkel, Affine Kac-Moody algebras and semi-infinite flag manifolds,
 Comm. Math. Phys. {\bf 128} (1990) 161-189.

\bibitem[FFR]{FFR} B. Feigin, E. Frenkel and N. Reshetikhin, Gaudin Model, Bethe Ansatz and Critical
Level, Comm. Math. Phys. 166 (1994) 27-62.


\bibitem[FG]{FG} E. Frenkel, D. Gaitsgory, Local Geometric Langlands Correspondence: the Spherical Case, in: Algebraic Analysis and Around: In Honor of Professor Masaki Kashiwara's 60th Birthday, in: Adv. Stud. Pure Math., vol. 54, 2009, p. 167,  arXiv:0711.1132

\bibitem[FP]{FP} L. Feher,  B.G. Pusztai, Explicit description of twisted Wakimoto realizations of affine Lie algebras, Nucl.Phys. B674 (2003) 509-532, arxiv: math/0305268



\bibitem [FHL]{FHL}
I. B. Frenkel, Y.-Z. Huang and J. Lepowsky, On axiomatic
approaches to vertex operator algebras and modules, Mem. Amer.
Math. Soc. {\bf 104}, 1993.

\bibitem [FLM]{FLM}
I. B. Frenkel, J. Lepowsky and A. Meurman,   Vertex Operator
Algebras and the Monster,   Pure and Applied Math., Vol. {\bf
134}, Academic Press, New York, 1988.





\bibitem [FST]{FST} B. L. Feigin, A. M. Semikhatov and  I. Yu. Tipunin,
Equivalence between chain categories of representations of affine
$sl(2)$ and $N=2$ superconformal algebras, J. Math.  Phys. {\bf
39} (1998), 3865-390.

\bibitem[FZ]{FZ}
I. B. Frenkel and Y.  Zhu, Vertex operator algebras associated to
representations of affine and Virasoro algebras,  Duke Math. J.
{\bf 66} (1992),  123-168.


\bibitem [K1]{K-b} V. Kac, Infinite dimensional Lie algebras, Third edition, Cambridge Univ. Press, Cambridge, 1990.


\bibitem [K2]{K}  V. Kac,   Vertex Algebras for Beginners, University
Lecture Series, Second Edition,   Amer. Math. Soc., 1998, Vol. 10.

\bibitem [KK]{KK} V. Kac and D. Kazhdan, Structure of
representations with highest weight of infinite dimensional Lie
algebras, Adv. Math. {\bf 34} (1979) 97-108.


\bibitem [KS]{KS} Y. Kazama and H. Suzuki, New N=2 superconformal
field theories and superstring compactifications, Nucl. Phys. B
{\bf 321} (1989), 232-268.







\bibitem [LL]{LL} J. Lepowsky and H. Li, Introduction to vertex
operator algebras and their representations, Progress in Math.,
Vol. 227,  Birkh\"auser, Boston, 2004.



\bibitem[S]{Scz} M. Szczesny, Wakimoto modules for twisted affine
Lie algebras,  Math. Res. Lett. {\bf 9} (2002), no.4, 433-448.

\bibitem[W]{W-mod} M. Wakimoto, Fock representations of affine Lie
algebra $A_1^{(1)}$, Comm. Math. Phys. 104 (1986) 605-609.



\end{thebibliography}
\end{document}